\documentclass[10pt]{article}
\usepackage{amsmath}
\usepackage{amsthm}
\usepackage{amsfonts}
\usepackage{amssymb}
\usepackage{todonotes}
\usepackage{nicefrac}
\usepackage{dsfont}
\usepackage{soul}
\usepackage{mathtools}
\usepackage{verbatim}
\usepackage{enumitem}
\usepackage[backend=bibtex,url=false,isbn=false,style=alphabetic,maxnames=5,firstinits=true,doi=true]{biblatex}

\newbibmacro*{bbx:parunit}{
  \ifbibliography{
     \setunit{\bibpagerefpunct}\newblock
     \usebibmacro{pageref}
     \clearlist{pageref}
     \setunit{\adddot\par\nobreak}
  }{}
}

\renewbibmacro*{doi+eprint+url}{
  \usebibmacro{bbx:parunit}
  \iftoggle{bbx:doi}{\printfield{doi}}{}
  \iftoggle{bbx:eprint}{\usebibmacro{eprint}}{}
  \iftoggle{bbx:url}{\usebibmacro{url}}{}
}

\renewbibmacro*{eprint}{
  \usebibmacro{bbx:parunit}
  \iffieldundef{eprinttype}{\printfield{eprint}}{\printfield[eprint:\strfield{eprinttype}]{eprint}}
}
     
\renewbibmacro{in:}{}

\AtEveryBibitem{%
  \clearlist{language}%
  \clearfield{issn}%
  \clearfield{month}%
}
\DeclareNameAlias{default}{first-last}

\bibliography{references}

\theoremstyle{plain}
\newtheorem{theorem}{Theorem}

\newtheorem{lemma}{Lemma}
\newtheorem{proposition}{Proposition}
\newtheorem{remark}{Remark}
\newtheorem{definition}{Definition}

\newtheorem{example}{Example}

\newcommand{\name}[1]{#1} 
\newcommand{\notion}[1]{\textit{#1}}

\newcommand{\const}{\ensuremath{\text{const}}}
\newcommand{\circled}[1]{\ensuremath{\text{(#1)}}}
\newcommand{\rg}{\mathrm{rk}}

\begin{document}

\title{Reducibility of Valence-3 Killing Tensors in Weyl's Class of Stationary and Axially Symmetric Space-Times}
\author{Andreas Vollmer\\\small{Institut f\"ur Mathematik, Friedrich-Schiller-Universit\"at Jena (Germany)}\\\small{andreas.vollmer@uni-jena.de}}
\maketitle

\section*{Abstract}
Stationary and axially symmetric space-times play an important role in astrophysics, particularly in the theory of neutron stars and black holes.
The static vacuum sub-class of these space-times is known as Weyl's class, and contains the Schwarzschild space-time as its most prominent example.
This paper is going to study the space of Killing tensor fields of valence~3 for space-times of Weyl's class. Killing tensor fields play a crucial role in physics since they are in correspondence to invariants of the geodesic motion (i.e.\ constants of the motion).

It will be proven that in static and axially symmetric vacuum space-times the space of Killing tensor fields of valence~3 is generated by Killing vector fields and quadratic Killing tensor fields. Using this result, it will be proven that for the family of Zipoy-Voorhees metrics, valence-3 Killing tensor fields are always generated by Killing vector fields and the metric.

\section{Introduction}
Consider a manifold $M$ with Lorentzian metric $g$, and its cotangent space $T^\ast M$ endowed with its natural symplectic form.
A Killing tensor field~$K$ of valence~$d$ on~$M$ is a symmetric $(0,d)$-tensor such that
\begin{equation}\label{eqn:killing}
  \nabla_{(a}K_{b_1\dots b_d)}=0.
\end{equation}
The Lorentzian metric provides an isomorphism between $T^\ast M$ and $TM$, and we are therefore going to identify co- and contravariant tensor fields as well as the corresponding homomorphisms with mixed co- and contravariant indices.

Killing tensor fields are in 1-to-1 correspondence to (first) integrals (or \notion{Hamiltonian invariants}) polynomial in the momenta $p$ (or the velocities $\dot{\gamma}$) of the Hamiltonian motion for the Hamiltonian function $H=g^{ij}p_ip_j=g(\dot{\gamma},\dot{\gamma})$.
In the language of integrals, the Killing tensor equation takes the form
\begin{equation}\label{eqn:integral-killing}
  \{I_K,H\} = X_H(I_K) \equiv 0
\end{equation}
where $\{\cdot,\cdot\}$ denotes the usual \name{Poisson} bracket on $T^\ast M$, and where $X_H(I_K)$ is the derivative of the function $I_K=K(\dot{\gamma},\dots,\dot{\gamma})$ in the direction of the Hamiltonian vector field $X_H$.
Two integrals are said to be in involution if they commute with respect to the Poisson bracket

Studying the existence of polynomial integrals is interesting from at least two perspectives.
Firstly, the existence of integrals can help in answering natural questions about the behavior of trajectories, i.e.\ the behavior of free falling particles in physically motivated Hamiltonian systems (e.g.\ by the famous Liouville-Arnold Theorem one can integrate the system by quadratures under certain additional assumptions).

Secondly, asking for the existence of integrals is a natural geometric requirement. Metrics meeting this requirement may lead to physically interesting examples, as for example in the case of the \name{Kerr} metric possessing the \name{Carter} constant \cite{carter_hamilton-jacobi_1968, carter_global_1968}, an integral quadratic in momenta in addition to energy and axial symmetry. Integrals polynomial in the momenta are of particular interest since they represent a generalization of constants of the motion that emerge from the action of one-parameter groups of diffeomorphisms. According to \cite{markakis_constants_2014}, the example of \name{Kerr-de Sitter} space-times is at present the only known example of integrable space-times with an additional integral of higher-than-linear degree, among the class of stationary and axially symmetric space-times.

Static and axially symmetric vacuum (StAV) space-times form \notion{Weyl's class}; they are special cases of stationary and axially symmetric vacuum (SAV) space-times \cite{stephani_exact_2003, brink_II_2008}.
Recently, some attention has been drawn to the \name{Zipoy-Voorhees} family, which belongs to this class \cite{voorhees_static_1970, zipoy_topology_1966}. Numerical studies \cite{brink_I_2008, brink_II_2008, brink_III_2009, brink_IV_2009} suggested integrability for some StAV metrics, while later studies provided contradicting evidence \cite{kruglikov_nonexistence_2011, gerakopoulos_non-integrability_2012, maciejewski_nonexistence_2013}. Note that while these studies considered fixed values of the parameter $\delta$ of the \name{Zipoy-Voorhees} metric, in this paper we are going to consider arbitrary $\delta$ in the case of the \name{Zipoy-Voorhees} metric. Therefore the result for the \name{Zipoy-Voorhees} family is in line with the evidence contradicting integrability of the family.

The methods used in this paper are not restricted to the concrete setting of Weyl's class. It is therefore likely that the methods are suitable for other examples of parametrized metrics in two dimensions.
It seems probable that such examples include 2-dimensional ones with potential, since this is closest to the case studied here.
The study of integrals in 2-dimensional manifolds is a classical problem in differential geometry and goes at least back to Darboux and Koenigs \cite{darboux_cons_1887}.
For instance, some non-existence results have been obtained on the 2-torus for cubic and quartic integrals in \cite{byalyi_first_1987, denisova_polynomial_2000, bialy_cubic_2011}, while for higher degrees almost nothing is known.
On the 2-sphere, however, some integrable examples are known. For instance, a new integrable system has been presented, with the additional integral being of cubic degree, by Dullin and Matveev \cite{dullin_new_2004}.

The procedure taken in this paper is a new approach to the question of existence of integrals.
It combines two major previous lines of action:
\begin{itemize}
 \item[--] It is inspired by ideas from prolongation-projection methods well-known in the theory of overdetermined PDE systems.
 However, it does not follow the algorithmic procedure used in \cite{kruglikov_nonexistence_2011}.
 \item[--] It follows ideas outlined in \cite{hietarinta_direct_1987}, but takes a somehow converse track that could be described as ``bottom-up" in contrast to the ``top-down" approach taken in \cite{hietarinta_direct_1987}. The advantage of this direction of reasoning is that it avoids solving the leading-degree Poisson equation. Instead, our approach rather begins with solving simple geometric orthogonality relations.
\end{itemize}

\subsection{Static and Axially Symmetric Vacuum Space-Times}
Stationary and axially symmetric vacuum space-times possess two commuting Killing vector fields, one being space-like and the other being time-like.
Such space-times can be brought into the following standard form~\cite{stephani_exact_2003} by the aid of suitable coordinate transformations. The coordinates are called \name{Lewis-Papapetrou} coordinates \cite{brink_I_2008, brink_IV_2009}.
\begin{equation}\label{eqn:sav-metric}
 g=e^{2U}\ \left( e^{-2\gamma}\,\left(dx^2+dy^2\right)\,+R^2\,d\phi^2 \right) -e^{-2U}\,\left(dt-\omega\,d\phi\right)^2
\end{equation}
We will restrict our attention to vacuum space-times and therefore require the Ricci tensor of $g$ to be identically zero.
This is a fair assumption for the movement of test particles around astrophysical objects as long as electro-magnetic fields are ignored.
For SAV space-times, Ricci flatness is encoded in a set of equations which are called the \name{Ernst} equations.
In the static case, we require $\omega=0$. Then, the \name{Ernst} equations read as follows
\begin{align*}
 R_y\,U_y +R_x\,U_x +R\,U_{yy} +R\,U_{xx}	&= 0 \\
 \Delta R = R_{xx} + R_{yy}				&= 0 \\
 2R\,U_x^2-2R_y\,\gamma_y+2R_x\,\gamma_x+R_{xx}-R_{yy}-2R\,U_y^2 &= 0 \\
 2R\,U_x\,U_y+R_y\,\gamma_x+R_x\,\gamma_y+R_{xy}	&= 0
\end{align*}

The equations break up into two sets of two equations each, which we shall refer to as \notion{primary} and \notion{secondary} equations.
The primary equations give restrictions on $U$ and $R$. Provided $R$ is non-constant, $\Delta R=0$ allows setting $R=x>0$ by a change of coordinates~\cite{stephani_exact_2003}. 
If $R$ is constant, this change of coordinates is impossible, but one can show that $\Delta\gamma=0$ holds and the metric is flat.
In case of non-constant $R$, the secondary equations enable us to express derivatives of $\gamma$ in terms of derivatives of $U$, allowing us to eliminate them, and finally $\gamma$, from the equations.
We obtain the relations:
\begin{equation}\label{eqn:ernst}
\begin{aligned}
 U_{yy}			&= -U_{xx}-\frac{1}{x}\,U_x 
 \qquad &\qquad  \gamma_x		&= -x\,U_x^2+x\,U_y^2 \\
 R				&= x
 \qquad &\qquad  \gamma_y		&= -2x\,U_x\,U_y
\end{aligned}
\end{equation}

By definition, stationarity and axial symmetry can be described by the global symmetry group $h=\mathds{R}\times\mathds{S}^1$. The action of $h$ is \notion{Hamiltonian} and we denote the moment map by $\mu$.
$h$ acts freely, and we may pass to the symplectic quotient $Q_\text{red}=\mu^{-1}(0)/h$, which inherits a symplectic form from the initial space-time (with an additional compactness assumption, this is the \notion{\name{Marsden-Weinstein} quotient}, see e.g.~\cite{mcduff_introduction_1995}). 
In \name{Lewis-Papapetrou} coordinates,~$h$ acts along coordinate directions and we will be able to identify the reduced coordinates easily.
The 4-dimensional problem then is reformulated as a 2-dimensional problem with metric $g_\text{red}$, and the Hamiltonian $H=T+V$ splits into a kinetic term $T=H_\text{red}=g_\text{red}^{ij}p_ip_j$ along with a potential $V$, which is polynomial in $p_\phi$ and $p_t$.
As a consequence, we distinguish \notion{Hamiltonian integrals} that commute with $H$, from \notion{metric integrals} that commute with $T$ only.
Note that the highest-degree component w.r.t.\ $(p_x,p_y)$ of a Hamiltonian integral is metric (i.e.\ commutes with $T$).
The metric and the potential on the reduced space read:
\begin{subequations}\label{eqns:reduced-space-metrics}
\begin{align}
 g_\text{red}	&= e^{2U}\ \left( e^{-2\gamma}\,\left(dx^2+dy^2\right) \right) \\
 V				&=R^{-2} e^{-2U}\,p_\phi^2  -e^{2U}\,p_t^2
\end{align}
\end{subequations}

\subsection{Main Results}

We consider integrals of the general form
\begin{equation}\label{eqn:form-integrals}
 I(x,y) = \sum_{i=0}^3\ \sum_{j=0}^i\ \sum_{k=0}^{3-i}\ a_{i,j,k}(x,y)\ p_x^j\,p_y^{i-j}\,p_\phi^k\,p_t^{3-i-k}
\end{equation}
Such integrals are in involution with $p_\phi$ and $p_t$.
Since the Hamiltonian defined by equations~\eqref{eqns:reduced-space-metrics} does not mix momenta $p_x,p_y$ with $p_\phi,p_t$, the components of~\eqref{eqn:form-integrals} of odd and even parity in $(p_x,p_y)$ can be considered separately.
We prove reducibility of degree-3 integrals in space-times of Weyl's class.
\begin{definition}
Let $I$ be a polynomial integral of degree~$d$. We say that $I$ is \emph{reducible} (by one degree) if there are polynomial integrals $I_1,\dots, I_m$ of degree at most $d-1$ such that $I$ is a linear combination of products of the integrals $I_i$.
We say that $I$ is \emph{totally reducible} if there is a representation of this form such that the $I_i$ are integrals linear in momenta or the Hamiltonian.
\end{definition}

For space-times in Weyl's class, we prove reducibility of degree~3 integrals by one degree.
In addition, total reducibility of degree-3 integrals is shown for the family of Zipoy-Voorhees space-times (a sub-family of Weyl's class).

\begin{theorem}\label{thm:main-result-1}
Let $M$ be a space-time in Weyl's class. Then any integral~\eqref{eqn:form-integrals} of third degree on $M$ is reducible.
\end{theorem}
Reducibility of a degree-3 polynomial integral means that these integrals can be written via products of lower-degree integrals. In the language of Killing tensors, this means that any valence-3 Killing tensor field can be written via symmetrized products of Killing vector fields and quadratic Killing fields.

For a concrete example, consider the \name{Zipoy-Voorhees} class of space-times.
Their metrics are static and axially symmetric, and parametrized by a parameter $\delta\geq 0$ \cite{voorhees_static_1970}.
Therefore, this family forms a subset of Weyl's class.
\begin{multline}
  g=\left(\frac{x+1}{x-1}\right)^\delta \Bigg(
		\left(\frac{x^2-1}{x^2-y^2}\right)^{\delta^2-1} dx^2
		+\left(\frac{x^2-1}{x^2-y^2}\right)^{\delta^2}\frac{x^2-y^2}{1-y^2}\ dy^2 \\
		+(x^2-1)(1-y^2)\ dz^2 \Bigg)  -\left(\frac{x-1}{x+1}\right)^\delta dt^2.
\end{multline}
The resulting metric for $\delta=0$ is flat. The value $\delta=1$ gives the Schwarzschild metric (that admits one additional quadratic integral).
We allow arbitrary $\delta\geq0$.

\begin{proposition}\label{prop:main-result-2}
Let $M_\text{ZV}$ be a Zipoy-Voorhees metric with parameter $\delta>0$, $\delta\not=1$. Let~$I$ be an integral~\eqref{eqn:form-integrals} of third degree on $M_\text{ZV}$. Then $I$ is totally reducible, i.e.\ generated by linear integrals (i.e.\ Killing vector fields) and the Hamiltonian (i.e.\ the metric).
\end{proposition}

\section{Method}\label{sec:method}

The basic procedure is as follows:
\begin{enumerate}[label=(\roman*)]
\item Reduce the 4-dimensional problem without potential to finding integrals on a 2-dimensional Hamiltonian manifold with potential (symplectic reduction).
\item The existence of integrals is encoded in equations that emerge from the \name{Poisson} equation $\{H,I_K\}=0$ as coefficients w.r.t.\ momenta. Splitting according to the degree in momenta $(p_x,p_y)$ yields three polynomials in $p_\phi$ and $p_t$. If we decompose these polynomials further w.r.t.\ momenta, we obtain three blocks of equations.
\item Use the equations obtained from zeroth degree in momenta $(p_x,p_y)$ and solve them as far as possible, obtaining one function $\alpha$ to parametrize the integral (this is the case without an additional integral present. If there is an additional linear integral, lemma~\ref{la:rk1} applies).
\item From the block obtained from the degree-2 polynomial, extract two integrability conditions for $\alpha$.
\item The remaining system of equations is an overdetermined system of PDE involving the metric which is described by one function $U$ in two variables. We consider derivatives of $U$ as being new, independent unknowns. By taking derivatives (prolongation), and then eliminating higher derivatives of~$U$ (projection), we end up with an ordinary differential equation.
\item For the remaining ODE we show that its only solution corresponds to flat space. This allows to conclude that degree-3 integrals are always reducible.
\end{enumerate}
%

Equation \eqref{eqn:integral-killing} is the condition for a function $I$ to be an integral. Since we take $I$ to be polynomial in momenta of degree~$d$, \eqref{eqn:integral-killing} is a polynomial in momenta of degree~$d+1$. We are going to consider the system of PDE obtained from the coefficients of \eqref{eqn:integral-killing} w.r.t.\ momenta.
Symplectic reduction w.r.t.\ the symmetry group (stationarity, axial symmetry) suggests to regard $p_\phi$ and $p_t$ as parameters. We distinguish the equations of the PDE system by the momenta monomials to which they appeared as a coefficient.
For Weyl's class, the equations can then be arranged in a tree-like\cite{brink_II_2008} structure. We write down the Hamiltonian in the following form:
\begin{equation}\label{eqn:split-H}
 H = T + \underbrace{V^{\phi\phi}p_\phi^2 + V^{tt}p_t^2}_{=V}
\end{equation}
where $T\equiv H_\text{red}$ is the reduced Hamiltonian (i.e.\ a homogeneous polynomial in $p_x$, $p_y$, and where $V^{ab}$ are the smooth coefficient functions of $p_a\,p_b$ ($a,b\in\{\phi,t\}$)).
The integral $I$ can be decomposed accordingly. We denote
\begin{equation}\label{eqn:split-I}
 I=I^\circled{d}
 		+\underbrace{ I^\circled{d-1}_\phi p_\phi +I^\circled{d-1}_t p_t }_{=I^\circled{d-1}}
 		+\underbrace{ I^\circled{d-2}_{\phi\phi} p_\phi^2
 									+I^\circled{d-2}_{t\phi} p_t p_\phi
 									+I^\circled{d-2}_{tt} p_t^2 }_{=I^\circled{d-2}}
 		+\dots
 		+I^\circled{0}_{tt\dots t} p_t^d,
\end{equation}
where each $I^\circled{k}$ is of degree~$k$ in the momenta $p_x,p_y$.
We require the metric to be non-flat such that we can choose coordinates with $R=x$.
In this case we have three blocks of equations coming from the respective polynomials (this is step (ii) of the above list). We can extract the equations from the polynomials which are obtained by splitting~\eqref{eqn:integral-killing} according to the degree w.r.t.\ $(p_x,p_y)$:%
\begin{subequations}\label{eqn:blocks}
\begin{align}
 \{T,I^\circled{3}\}	&=0
 && \text{\textit{degree~4}}\label{eqn:key-block} \\
 \{T,I^\circled{1}\} +\{V,I^\circled{3}\} &=0
 && \text{\textit{degree~2}}\label{eqn:middle-block} \\
 \{V,I^\circled{1}\}	&=0
 && \text{\textit{degree~0}}\label{eqn:bottom-block}
\end{align}
\end{subequations}
The equations of even parity in $(p_x,p_y)$ split off from this system and form a separate, decoupled system.
Equation \eqref{eqn:key-block} is the condition that must hold for an integral $I^\circled{3}$ on the reduced manifold with Hamiltonian $T=H_\text{red}$.
However, only some of these integrals ascent to integrals upstairs on the initial manifold.
This is due to the restrictions \eqref{eqn:middle-block} and \eqref{eqn:bottom-block}.
For a better understanding of these equations (or the equations obtained from them as coefficients w.r.t.\ momenta), we will characterize them as defining equations for $I^\circled{3}$ and $I^\circled{1}$.
But first let us split the system further by considering coefficients w.r.t.~$(p_t, p_\phi)$.
The polynomial \eqref{eqn:key-block} does not split since it is already of degree~4 in momenta $(p_x,p_y)$.
The polynomial \eqref{eqn:middle-block} splits into three parts:
\begin{align*}
 \{T,I^\circled{1}_{\phi\phi}\} +\{V^{\phi\phi},I^\circled{3}\} &=0 \\
 & 0 = \{T,I^\circled{1}_{t\phi}\} \\
 \{T,I^\circled{1}_{tt}\} +\{V^{tt},I^\circled{3}\} &=0
\end{align*}
We write the equations in this form to hint at the fact that the equations can be divided into two groups that can be treated separately.
This will become clear when we include \eqref{eqn:bottom-block}.
The second of the equations says that $I^\circled{1}_{t\phi}$ is a metric integral on the reduced space.
In fact, we will see that it even has to be an integral on the initial space-time, and therefore is not of interest for our considerations.
The polynomial \eqref{eqn:bottom-block} splits into five parts:
\begin{align*}
 \{V^{\phi\phi},I^\circled{1}_{\phi\phi}\}	&=0 \\
 & 0 = \{V^{\phi\phi},I^\circled{1}_{t\phi}\} \\
 \{V^{tt},I^\circled{1}_{\phi\phi}\} + \{V^{\phi\phi},I^\circled{1}_{tt}\}	&=0 \\
 & 0 = \{V^{tt},I^\circled{1}_{t\phi}\} \\
 \{V^{tt},I^\circled{1}_{tt}\}	&=0
\end{align*}
The second and fourth of these relations tell us that $I^\circled{1}_{t\phi}$ is an integral not only on the reduced, but also on the initial space.
We can isolate this subsystem from the remaining one and solve it separately (this procedure is possible in general for Weyl's class). This can easily be done and is equivalent to finding Killing vector fields of the space-time under consideration.

The remaining equations from \eqref{eqn:bottom-block} can be interpreted in a nice way as scalar product relations for the components of $I^\circled{1}$. For instance,
\[
 \{V^{\phi\phi},I^\circled{1}_{\phi\phi}\} =
    V^{\phi\phi}_x\ b^{\phi\phi}_1 + V^{\phi\phi}_y b^{\phi\phi}_2 =
       e^{2U-2\gamma}\,\langle \nabla V^{\phi\phi}, b^{\phi\phi} \rangle = \langle dV^{\phi\phi},b^{\phi\phi}\rangle
\]
where $I^\circled{1}_{\phi\phi}=b_1^{\phi\phi} p_x+ b_2^{\phi\phi} p_y$ and where $\nabla V^{\phi\phi}$ denotes the gradient vector corresponding to the differential $dV^{\phi\phi}$.
The polynomial \eqref{eqn:bottom-block} therefore gives rise to a set of scalar product relations:
\begin{align*}
 \langle \nabla V^{\phi\phi}, b^{\phi\phi} \rangle &= 0 \\
 \langle \nabla V^{tt}, b^{\phi\phi} \rangle + \langle \nabla V^{\phi\phi}, b^{tt} \rangle &= 0 \\
 \langle \nabla V^{tt}, b^{tt} \rangle &= 0
\end{align*}
This allows to solve \eqref{eqn:bottom-block} directly for $b^{\phi\phi}$ and $b^{tt}$,
\[
 b^{\phi\phi} = \alpha_1\,\nabla^\perp V^{\phi\phi}, \qquad
 b^{tt} = \alpha_2\,\nabla^\perp V^{tt}
\]
where we introduce the shorthand notation
$\nabla^\perp f=e^{2U-2\gamma}(-f_y,f_x)$ for a function $f$,
i.e.\ $\nabla^\perp f$ is the vector field rotated by $\nicefrac{\pi}{2}$ compared to $\nabla f$.
Defining the angle $\Psi$ between $\nabla V^{tt}$ and $\nabla V^{\phi\phi}$,
\[
 \cos\Psi =\frac{\langle\nabla V^{tt},\nabla V^{\phi\phi}\rangle}{\lVert\nabla V^{\phi\phi}\rVert\ \lVert\nabla V^{tt}\rVert},
\]
the second of the three scalar product relations can be brought into the form
\begin{equation}\label{eqn:alpha-sin}
 (\alpha_2-\alpha_1)\ \sin\Psi = 0
\end{equation}
This is step (iii) in the list given at the beginning of this section. We summarize:
\begin{lemma}
 Either the metric potentials are such that $\nabla V^{\phi\phi}$ and $\nabla V^{tt}$ are parallel, or the parameter functions $\alpha_1$ and $\alpha_2$ are equal.
\end{lemma}

We now turn to an interpretation of \eqref{eqn:middle-block}. Consider $\{V,I^\circled{3}\}$ and denote $I^\circled{3} = I^{ijk} p_i p_j p_k$.
Then,
\begin{align*}
\{V,I^\circled{3}\}
 &= 3\left(V_x I^{xij} p_i p_j +V_y I^{yij} p_i p_j\right) \\
 &= 3\,(V_k I^{kij} p_i p_j),
\end{align*}
and analogously for $\{V^{\phi\phi},I^\circled{3}\}$ and $\{V^{tt},I^\circled{3}\}$.
With this in mind, we interpret \eqref{eqn:middle-block} as defining equations for the tensor field $K^\circled{3}(\nabla V,\cdot,\cdot)$.
There are two more equations than components of $K^\circled{3}$ and this allows us to find independent expressions for $K^\circled{3}(\nabla V^{\phi\phi},\cdot,\cdot)$ as well as $K^\circled{3}(\nabla V^{tt},\cdot,\cdot)$.
Then, if $dV^{tt}$ and $dV^{\phi\phi}$ are linearly independent, we can determine $K^\circled{3}$ in terms of derivatives of the function $\alpha= \alpha_1= \alpha_2$.
We have the following obvious identities:
\begin{subequations}\label{eqn:compatibilitaet-a_i}
\begin{align}
 K^\circled{3}(\nabla V^{\phi\phi};\nabla V^{\phi\phi},\nabla V^{tt})
 &=  K^\circled{3}(\nabla V^{tt};\nabla V^{\phi\phi},\nabla V^{\phi\phi}) \\
  K^\circled{3}(\nabla V^{\phi\phi};\nabla V^{tt},\nabla V^{tt})
  &=  K^\circled{3}(\nabla V^{tt};\nabla V^{\phi\phi},\nabla V^{tt})
\end{align}
\end{subequations}
We proceed as follows:
\begin{enumerate}
 \item Determine $K^\circled{3}$ in terms of $\alpha$ and its derivatives, if $\sin\Psi\not=0$.
 \item Determine derivatives of $\alpha$ using the symmetry in the arguments of $K^\circled{3}$. Then derive an integrability condition for $\alpha$.
 \item Combine the integrability condition with \eqref{eqn:key-block} and the \name{Ernst} equations. Show that the system does not have any solutions, using algebraic manipulations as well as prolongation-projection arguments for the system.
\end{enumerate}

\noindent First, however, consider the case $\sin\Psi=0$. In order to do this, we begin with a closer look at Killing vectors.

\subsection{Killing Vector Fields}
Assuming there is an additional linear integral on the 4-dimensional space-time, we characterize the existence of linear integrals in terms of the rank of the $2\times 3$-matrix whose columns are given by gradients of the potential components $V^{\phi\phi}$, $V^{t\phi}$ and $V^{tt}$:
\[
   \mathcal{M} = (dV^{\phi\phi},dV^{tt},dV^{t\phi}).
\]
Since the rank of $\mathcal{M}$ is the dimension of the linear space spanned by $dV^{\phi\phi}$, $dV^{tt}$, $dV^{t\phi}$, it is a geometric object and independent of the choice of coordinates.

If $dV^{t\phi}=0$, $\mathcal{M}$ might be replaced by the $2\times 2$-matrix $(dV^{\phi\phi},dV^{tt})$, which will also be denoted by $\mathcal{M}$.
Then, instead of the rank of the matrix, the determinant may be used with the obvious correspondences.
In case there is an additional linear integral present in a non-flat SAV space-time (i.e.\ in addition to $p_\phi$ and $p_t$), the rank of $\mathcal{M}$ cannot be full.
More precisely:
\begin{lemma}\label{la:add-lin}
 \begin{enumerate}
  \item[(a)] Let $(M,g)$ be in the SAV class.
     \begin{itemize}
       \item If there is an additional linear integral (Killing vector field), then the rank of $\mathcal{M}$ is~1, or the space-time is flat.
       \item Let $\rg(\mathcal{M})=1$. Then $p_y$ is a linear integral (Killing vector field) when using \name{Lewis-Papapetrou} coordinates with $R=x$.
     \end{itemize}
  \item[(b)] Let $(M,g)$ be in Weyl's class.
     \begin{itemize}
       \item Let $\rg(\mathcal{M})\leq 1$ be constant. Then there is an additional linear integral on $M$.
             In case $\rg(\mathcal{M})=1$ this vector field corresponds to $p_y$ in \name{Lewis-Papapetrou} coordinates with $R=x$; in case $\rg(\mathcal{M})=0$ the space-time is flat.
       \item If there is an additional linear integral, it is given by $p_y$ in \name{Lewis-Papapetrou} coordinates with $R=x$, if $M$ is non-flat.
     \end{itemize}
 \end{enumerate}
\end{lemma}

\begin{proof}
Part (a).
For linear integrals we only have two polynomials after taking coefficients w.r.t.\ $(p_x,p_y)$ (they are similar to the polynomials of degree~0 and~2 in \eqref{eqn:blocks}).
Let us denote the components of the $(p_x,p_y)$-linear part of the integral as
\begin{equation}\label{eqn:split-b-kv}
 I^\circled{1} = b_1\,p_x+b_2\,p_y, \qquad b=(b_1,b_2)
\end{equation}
The zeroth order equations read:
\begin{equation}
 \langle \nabla\,V^{\phi\phi}, b \rangle	= 0, \qquad
 \langle \nabla\,V^{t\phi}, b \rangle	= 0, \qquad
 \langle \nabla\,V^{tt}, b \rangle = 0
\end{equation}
We conclude that the following relations must hold:
\begin{equation}\label{eqn:kv-orthogonalities}
 \langle \nabla\,V^{\phi\phi}, \nabla^\perp V^{t\phi} \rangle	= 0, \qquad
 \langle \nabla\,V^{\phi\phi}, \nabla^\perp V^{tt} \rangle = 0, \qquad
 \langle \nabla\,V^{t\phi}, \nabla^\perp V^{tt} \rangle	= 0
\end{equation}
This means, the potential gradients are pointing all in the same direction, i.e.\ they are pairwise linearly dependent.
Hence, the rank of the potential gradient matrix is~1, provided the space-time is non-flat. In the flat case we may have \name{Lewis-Papapetrou} coordinates with the parameter $R$ constant, making it impossible to choose $R=x$. We therefore exclude the flat case from our considerations.
This concludes the proof of claim one and establishes a necessary criterion for the existence of Killing vector fields in everywhere non-flat space-times.
\vskip 0.3cm
Now, since the rank of the potential gradient matrix $\mathcal{M}$ is~1, all rows as well as all columns have to be linearly dependent.
This again gives us relations \eqref{eqn:kv-orthogonalities}, meaning that $\nabla\,V^{\phi\phi}$, $\nabla\,V^{t\phi}$ and $\nabla\,V^{tt}$ are pairwise linearly dependent.
First, let us assume $\omega\not=0$. We consider
\[
 \langle \nabla\,V^{\phi\phi}, \nabla^\perp V^{t\phi} \rangle =0
\]
This equation amounts to the requirement
\[
 x\,\omega_x\,U_y -(1+x\,U_x)\,\omega_y =0
\]
or the relation
\[
 \left(\begin{array}{c} \omega_x \\ \omega_y \end{array}\right)
 = \kappa\ \left(\begin{array}{c} 1+x\,U_x \\ x\,U_y \end{array}\right).
\]
with a scalar function $\kappa$ to be determined.
Inserting this into the requirement
\[
 \langle \nabla\,V^{\phi\phi}, \nabla^\perp V^{tt} \rangle =0
\]
yields the relation
\[
 U_y\,x^2\,e^{4U} = 0
\]
and forces $U$ to be a function of $x$ only.
Turning back to the relations for $\omega$, we see that $\omega_y=0$, so $\omega$ also is a function of $x$ only.

Recalling the convention $R=x$, the metric depends on $x$ only, if $\gamma$ only depends on $x$, or if it is constant.
We infer the secondary \name{Ernst} equations,
\begin{align*}
 4x^2\,e^{4U}\,U_x^2-\omega_x^2-4x\,e^{4U}\,\gamma_x	&= 0, \\
 \gamma_y\ x\,e^{4U}																	&= 0.
\end{align*}
Consider the latter equation. It means $\gamma_y=0$, so we are done.
Since the metric does not depend on $y$, $p_y$ must be an integral, and therefore provides a Killing vector field.

Now, assume $\omega=0$. Then $\langle\nabla V^{\phi\phi},\nabla^\perp V^{t\phi}\rangle=0$ trivially and we have to go a slightly different way of reasoning. We consider
\[
  \langle\nabla V^{\phi\phi},\nabla^\perp V^{tt}\rangle=0.
\]
It follows that
\[
  x^2\,e^{4U}\,U_y =0,
\]
which means $U_y=0$ (on the entire neighborhood). Conclude $U=U(x)$, and then $\gamma=\gamma(x)$.
Thus, the metric is a function of $x$ only and $p_y$ is a linear integral.
This concludes the proof of part (a).
\vskip 0.5cm
For part (b) let us first remark that if an additional linear integral exists in Weyl's class, it must be a multiple of $p_y$, or the metric is flat.
Two cases need to be checked:
Firstly, if there is exactly one additional linear integral, it is a multiple of $p_y$. Secondly, if there are two (independent) additional linear integrals, there are three (say $b^{(k)},k=1,2,3$). Looking at the equations $\langle\nabla\,V^{ij},b^{(k)}\rangle=0$, this forces all gradients $\nabla V^{ij}$ to be zero (or, equivalently, $dV^{ij}=0$). Hence, $V$ is constant. Thus $U$ and $\omega$ are constant, in contradiction to the assumption $R=x$. Therefore, the metric is flat.
\vskip 0.3cm
With this remark, the first claim of part (b) follows immediately from part~(a), keeping in mind that rank~0 corresponds to flat space.
The second claim of part~(b) follows immediately from the second statement of part~(a).
\end{proof}

\section{Proof of the Main Theorem}
For the proof we will w.l.o.g.\ assume constant rank for the matrix $\mathcal{M}$.
If the space-time $M$ is not of constant $\rg\,\mathcal{M}$, we may still consider the subsets of points in $M$ with constant rank $0$, $1$, or $2$.
We may then work with the sets of their inner points ignoring the remaining points of $M$, which amount only to a null set w.r.t.\ the measure induced by the volume form on $M$.
Proving the theorem on a dense set is sufficient because if a degree-3 polynomial integral is identical to a linear combination of products of $H$, $p_\phi$ and $p_t$ on an open subset, this is true everywhere.

The proof will be completed in two steps.
First we consider space-times with $\rg\,\mathcal{M}=1$, then the case $\rg\,\mathcal{M}=2$.
As we have seen, the \mbox{rank-1} case is the case when there is one additional Killing vector field. Rank-2 is the case if no additional Killing vector field exists (assuming non-flatness).

\begin{lemma}\label{la:rk1}
If $\rg\,\mathcal{M}=1$, then any third degree integral is reducible by at least one degree.
\end{lemma}
\begin{proof}
By the hypothesis, there is the linear integral $p_y$ in \name{Lewis-Papapetrou} coordinates with $R=x$.
Consider \eqref{eqn:middle-block}
\begin{align*}
 \{V^{\phi\phi},F^\circled{3}\}+\{T,F^\circled{1}_{\phi\phi}\}		&= 0 \\
 \{V^{t\phi},F^\circled{3}\}+\{T,F^\circled{1}_{t\phi}\}					&= 0 \\
 \{V^{tt},F^\circled{3}\}+\{T,F^\circled{1}_{tt}\}								&= 0
\end{align*}
Each $F^\circled{1}$ is a multiple of $p_y$, so
\[
 F^\circled{1}_{\phi\phi}=h_1\,p_y,\qquad
 F^\circled{1}_{t\phi}=h_2\,p_y,\qquad
 F^\circled{1}_{tt}=h_3\,p_y.
\]
This means that the equations in \eqref{eqn:middle-block} are of the form
\begin{align*}
 \{V^{\phi\phi},F^\circled{3}\}+\{T,h_1\}\,p_y	&= 0 \\
 \{V^{t\phi},F^\circled{3}\}+\{T,h_2\}\,p_y			&= 0 \\
 \{V^{tt},F^\circled{3}\}+\{T,h_3\}\,p_y				&= 0
\end{align*}
The leading order term $F^\circled{3}$ hence is of the form
\[
 F^\circled{3} = p_x\ ((\dots)\,p_y)+f\,p_y^3 =: F\,p_y
\]
where the leading $p_x$ is because the potential gradients (or, equivalently, the differentials $dV^{ab}$) have only $p_x$ components. The final contribution $f\,p_y^3$ accounts for the fact that \eqref{eqn:middle-block} only specifies components with at least one $p_x$.

\noindent Now consider \eqref{eqn:key-block},
\[
 \{T,F^\circled{3}\} = \{T,F\,p_y\} = \{T,F\}\,p_y \stackrel{!}{=} 0.
\]
This means $\{T,F\}=0$, so $F$ is a quadratic integral on the reduced space. It follows that it can be extended to an integral on the initial space-time, because of the fact that
\[
 \{V^{\phi\phi},F^\circled{3}\} = \{V^{\phi\phi},F\,p_y\} = \{V^{\phi\phi},F\}\,p_y
\]
and so on, so we have from \eqref{eqn:middle-block} the equations
\begin{align*}
 \{V^{\phi\phi},F\}+\{T,h_1\}	&= 0 \\
 \{V^{t\phi},F\}+\{T,h_2\}			&= 0 \\
 \{V^{tt},F\}+\{T,h_3\}				&= 0
\end{align*}
which makes $\tilde{F}=F+h_1\,p_\phi^2+h_2\,p_\phi\,p_t+h_3\,p_t^2$ a quadratic integral on the initial space-time (see remark~\ref{rmk:quadratic-integrals} below).
Note that $\tilde{F}$ might still be reducible, but can be non-reducible as well.
\end{proof}

\begin{remark}\label{rmk:quadratic-integrals}
An even-parity quadratic integral
$I=I^\circled{2}+I^\circled{0}_{\phi\phi}\,p_\phi^2
								+I^\circled{0}_{t\phi}\,p_t\,p_\phi
								+I^\circled{0}_{tt}\,p_t^2$
satisfies the polynomial equations
\begin{subequations}\label{eqns:quadratic-integrals}
\begin{align}
 \{T,I^\circled{2}\}	&= 0 \\[0.2cm]
 \{V^{\phi\phi},I^\circled{2}\}+\{T,I^\circled{0}_{\phi\phi}\}	&= 0 \\
 \{V^{t\phi},I^\circled{2}\}+\{T,I^\circled{0}_{t\phi}\}	&= 0 \\
 \{V^{tt},I^\circled{2}\}+\{T,I^\circled{0}_{tt}\}	&= 0
\end{align}
\end{subequations}
\end{remark}
\begin{proof}
Decompose $\{H,I\}=0$ by setting each component homogeneous in $(p_x,p_y)$ zero. The first equation is the component of degree~3, the other three equations are  components of degree~1.
\end{proof}

We now turn to the case when there is no additional linear integral in involution with the others.
From now on, we will always assume to work in Weyl's class.

Keeping in mind the considerations of the previous sections, we see that this case requires $\rg\,\mathcal{M}=2$.
Rank~2 requires $\nabla\,V^{\phi\phi}$ and $\nabla\,V^{tt}$ to be linearly independent.
Then, recalling equation \eqref{eqn:alpha-sin}, the scaling functions $\alpha_1$ and $\alpha_2$ are equal for Weyl's class.
For simplicity we therefore introduce the new function $\alpha=\alpha_1=\alpha_2$ (step (iii) in the list of section~\ref{sec:method}), so
\[
  b^{\phi\phi}=\alpha\,\nabla^\perp\,V^{\phi\phi}, \qquad 
  b^{tt}=\alpha\,\nabla^\perp\,V^{tt}.
\]

\begin{lemma}
Derivatives of $\alpha$ are determined by differential equations of the form
\begin{align*}
 \alpha_x	&= A\,\alpha \\
 \alpha_y	&= B\,\alpha
\end{align*}
where $A$ and $B$ are algebraic expressions determined by $V^{tt}_x$, $V^{tt}_y$, $V^{\phi\phi}_x$ and $V^{\phi\phi}_y$, which do not contain any higher-than-second derivatives of components of the potential $V$.
\end{lemma}
\begin{proof}
We use the relations \eqref{eqn:compatibilitaet-a_i}, i.e.\ we use the six equations from \eqref{eqn:middle-block} and combine them in a straightforward way to find expressions for the coefficients $a_0$ through to $a_3$ of $I_T=\sum_i a_i p_x^{d-i}p_y^i$.
In this way, we find two different expressions for~$a_1$, and two for~$a_2$, corresponding to the above identities.
The expressions are polynomials in derivatives of the potential $V$, i.e.\ they are determined by $V^{tt}_x$, $V^{tt}_y$, $V^{\phi\phi}_x$ and $V^{\phi\phi}_y$ and do not contain derivatives of order higher than~2.
The coefficients of the $a_i$ are just integer multiples of $\nu=\langle\nabla\,V^{tt},\nabla^\perp\,V^{\phi\phi}\rangle$, which is non-zero because we required $\nabla\,V^{tt}$ and $\nabla\,V^{\phi\phi}$ to be linearly independent.

We can then eliminate $a_1$ and $a_2$ and deduce two equations which have the following form:
\begin{align*}
\langle\nabla\,V^{tt},\nabla^\perp V^{\phi\phi}\rangle\ \alpha_x
 &= (\dots)\,\alpha \\
\langle\nabla\,V^{tt},\nabla^\perp V^{\phi\phi}\rangle\ \alpha_y
 &= (\dots)\,\alpha
\end{align*}
The expressions abbreviated by $(\dots)$ are polynomials in derivatives of $V$ of at most second order.
Dividing by the non-zero coefficient of the $\alpha$-derivatives yields the required result.
\end{proof}
The integrability condition for $\alpha$ is a necessary requirement for the existence of non-reducible Killing tensor fields (step (iv) in the list of section~\ref{sec:method}):
\begin{lemma}\label{la:nec-crit}
Let $\rg\,\mathcal{M}=2$ and $\omega=0$, but $\nabla\,V^{\phi\phi},\nabla\,V^{tt}\not=0$. If there is a Killing tensor field of valence~3, then either $\alpha=0$ or $A_y-B_x=0$.
\end{lemma}
We note that in case $\alpha=0$ the integral $F_3=F^\circled{3}+F^\circled{1}=0$, so the lemma actually provides a necessary criterion for the existence of non-trivial Killing tensor fields of valence~3. 
\begin{proof}[Proof of lemma \ref{la:nec-crit}]
Compute
\begin{align*}
 (\alpha_x)_y-(\alpha_y)_x
 &= A_y\,\alpha+A\,\alpha_y -B_x\,\alpha-B\,\alpha_x \\
 &= (A_y-B_x)\,\alpha +(AB-BA)\,\alpha \\
 &= (A_y-B_x)\,\alpha
\end{align*}
\end{proof}
We give an example where this idea provides information on the reducibility of linear integrals:

\begin{example}
The Zipoy-Voorhees family of metrics is a family in Weyl's class that is parametrized by a non-negative number~$\delta$.

\noindent We can use the method as described above, but we take $H$ in a modified form, namely
\[
   H = \frac{p_x^2}{2\Omega_1}+\frac{p_x^2}{2\Omega_1}+V^{\phi\phi}\,p_\phi^2+V^{tt}\,p_t^2
\]
The Zipoy-Voorhees metric satisfies, in prolate spheroidal coordinates:
\begin{align*}
 \Omega_1	&= \frac{1}{2} \left(\frac{x^2-1}{x^2-y^2}\right)^{\delta^2} \left(\frac{x+1}{x-1}\right)^\delta \frac{x^2-y^2}{x^2-1} \\
 \Omega_2	&= \frac{1}{2} \left(\frac{x^2-1}{x^2-y^2}\right)^{\delta^2} \left(\frac{x+1}{x-1}\right)^\delta \frac{x^2-y^2}{1-y^2} \\
 V^\phi		&= \left( \left(\frac{x+1}{x-1}\right)^\delta\ (x^2-1)\,(1-y^2)\right)^{-1} \\
 V^t			&= -\left(\frac{x+1}{x-1}\right)^\delta
\end{align*}

Taking an approach similar to lemma \ref{la:nec-crit}, we first check that $\det\,\mathcal{M}\not=0$. We find the following:
\[
 \det\,\mathcal{M} = 0 \quad\Leftrightarrow\quad
 \delta^2\,y^2\,\frac{x^8-4x^6y^2+6y^4x^4-4y^6x^2+y^8}{(x-1)^2\,(x^2-1)^4\,(-1+y^2)^6\,(x+1)^2} = 0
\]
which obviously is not true for generic $x,y$, if $\delta\not=0$.
Then we compute the necessary criterion as in lemma~\ref{la:nec-crit}. Using computer algebra (Maple 18), we find
\[
 A_y-B_x \stackrel{!}{=} 0 \quad\Leftrightarrow\quad
 \frac{(x-y)^4\,(x+y)^4\,(-3\delta\,x^2+4\delta^2\,x+2x-3\delta)\,\delta^2\,y}{(x^2-1)^8\,(y^2-1)^6} \stackrel{!}{=} 0
\]
which is not true for generic $x,y$ since $\delta>0$. We therefore must conclude $\alpha=0$, which means the integral of degree~3 is zero.
\end{example}

We now turn to step (v) in the list at the begining of section~\ref{sec:method}:
\begin{lemma}\label{la:x-coeffs_polynomial}
A polynomial equation of degree $N>0$ for a function $f(x,y)$ with coefficients that depend on $x$ only, is independent of $y$, so $f=f(x)$.
\end{lemma}
\begin{proof}
Denote the equation by $\sum_{n=0}^N a_n(x)\,f^n(x,y)=0$.
If we can factor out $f(x,y)$, then $f=0$ and is independent of both $x$ and $y$.
Otherwise, we take one derivative w.r.t.\ $y$ and obtain
 $\sum_{n=1}^N a_n(x)\,n\,f^{n-1}\,f_y=0$.
Then either $f_y=0$ or we divide by $f_y$ and proceed similarly.
At some point, we either get $f_y=0$ or we end up with $a_N=0$, which contradicts the hypothesis that the polynomial equation be of degree~$N$.
Thus we have $f_y=0$ and $f$ is a function of $x$ only.
\end{proof}

\begin{lemma}\label{la:Ux(x)-means-Uy=0}
Let $U_x=U_x(x)$ be a function of $x$ only. Let the StAV space-time have a non-reducible third degree integral. Then $U_y=0$.
\end{lemma}
\begin{proof}
The proof has two parts:
(1) Show that $U_y$ has to be constant,
(2) Show that the constant is zero.\\
For the first part, consider the $p_1^3p_2$-component of \eqref{eqn:key-block}. Use the \name{Ernst} equation to substitute derivatives $U_{yy}$. In this way, obtain the equation:
\begin{multline}
  10x^3\,U_y^4 +156x^2\,_x\,(1+x\,U_x)\,U_y^2 +36x^2\,U_x\,U_{xx} \\
  -126x^3\,U_x^4 -126x\,U_x^2+18x\,U_{xx}-18\,U_x -252x^2\,U_x^3 =0
\end{multline}
This is a polynomial equation of degree~4 for $U_y$, and all coefficients are functions of $x$ only. 
By lemma~\ref{la:x-coeffs_polynomial}, this means $U_{yy}=0$, so $U_y=\const=:c$.

For the second part of the proof, we insert this result into the $p_1^4$-component of \eqref{eqn:key-block}. If we substitute $U_{xx}$ with the help of the \name{Ernst} equation, we find:
\[
  6\,U_x\,(1+x\,U_x)\,(1+2x\,U_x) =0
\]
Hence, there are 3~cases: $U_x=0$, $U_x=-\frac{1}{x}$ and $U_x=-\frac{1}{2x}$. We treat them separately:
\begin{itemize}
 \item If $U_x=0$, use again the $p_1^3p_2$-component which reads
    \[
      10x^3\,c^6 =0,
    \]
    so $c=0$.
 \item For $U_x=-\frac{1}{x}$ we have the same equation, so again $c=0$.
 \item The case $U_x=-\frac{1}{2x}$ is slightly more involved. The $p_1^3p_2$-component reads
    \[
      \frac{c^2\,(9-312x^2\,c^2+80x^4\,c^4)}{8x} =0.
    \]
    Now, either $c=0$ directly, or $9-312x^2\,c^2+80x^4\,c^4$. In the latter case, $xc=\const$ and hence $c=0$.
\end{itemize}
\end{proof}

\begin{lemma}\label{la:flat_metric_parametrisation}
Using \name{Lewis-Papapetrou} coordinates $(x,y)$, assume the potential function $U$ to be
\[
 e^{2U}=k_U\,\frac{2y+c+\sqrt{4x^2+4y^2+4cy+c^2}}{x^2}, \text{ with } k_U,c\in\mathds{R}.
\]
This provides a parametrization of flat space.
\end{lemma}
\begin{proof}
Determine the function $\gamma$ from the secondary \name{Ernst} equations and find
\[
  e^{2\gamma}=2\,k_\gamma\ \sqrt{4x^2+4y^2+4cy+c^2}\ e^{2U}, \qquad k_\gamma\in\mathds{R}.
\]
Then compute the Riemann curvature tensor for the metric
\[
  g= e^{2U}\,\left(e^{-2\gamma}\,\left( dx^2+dy^2\right)+x^2\,d\phi^2\right) -e^{-2U}\,dt^2
\]
and find that it vanishes identically.
Thus, the potential function $U$ defines a flat metric, which of course is StAV.
\end{proof}

\begin{lemma}\label{la:rk2}
Let $\rg\,\mathcal{M}=2$ and $\omega=0$, but $\nabla\,V^{\phi\phi},\nabla\,V^{tt}\not=0$. Assume $\alpha\not=0$.
Then there is no non-trivial Killing tensor of valence~3.
\end{lemma}
\begin{proof}
We assume there was such a Killing tensor. Then, by the necessary criterion (lemma~\ref{la:nec-crit}), $A_y-B_x=0$. In addition, consider \eqref{eqn:key-block}, and the \name{Ernst} equations. Since we chose \name{Lewis-Papapetrou} coordinates with $R=x$, we have $U_y\not=0$.

Consider \eqref{eqn:key-block} in combination with the necessary criterion from lemma~\ref{la:nec-crit}, plus the \name{Ernst} equations. The \name{Ernst} equations are to be invoked mainly in order to substitute $\frac{\mathrm{d}^2U}{\mathrm{d}y^2}$.
We take derivatives w.r.t.\ $x$ and $y$ of \eqref{eqn:key-block}. Then, we have 18~equations (\eqref{eqn:key-block} plus the necessary criterion $A_y-B_x=0$ from lemma~\ref{la:nec-crit}). Using the Ernst equations, we have only the following unknown functions:
\[
 U_{xxxx},\ U_{xxxy},\ U_{xxx},\ U_{xxy},\ 
 U_{xx},\ U_{xy},\ U_{x},\ U_{y},\ 
 U, \gamma.
\]
Use the $x$-derivative of the $p_1^3p_2$-component to substitute $U_{xxxy}$, and the $x$-derivative of the $p_1^2p_2^2$-component to substitute $U_{xxxx}$ in terms of lower order derivatives. The quantity $U_{xxy}$ can be substituted for via the $x$-derivative of the integrability criterion, but only if
\begin{equation}\label{eqn:seitenzweig}
 (1+2x\,U_x)\,(x\,U_x^2-3x\,U_y^2+U_x) \not=0.
\end{equation}
In this case, we can proceed as follows: Substitute $U_{xxx}$ by the $x$-derivative of the $p_1^4$-component, and use this component to substitute $U_{xx}$. Finally, substitute $U_{xy}$ using the integrability condition.

With all these substitutions at hand, we now have only equations in the unknowns $U_x$, and $U_y$ left.
For instance, the derivative w.r.t.\ $y$ of the $p_1^4$-component of \eqref{eqn:key-block} reads
\[
 x\,U_x^2\,(1+2x\,U_x)\,(1+x\,U_x)^2\,(x\,U_x^2+U_x+x\,U_y^2)^3 = 0.
\]
Therefore, either $U_x=0$ or $U_x=-\frac{1}{x}$ or $U_x=-\frac{1}{2x}$, or $x\,U_x^2+U_x+x\,U_y^2=0$. The three cases mentioned first are covered by lemma~\ref{la:Ux(x)-means-Uy=0}, and obviously in contradiction to the hypothesis $U_y\not=0$.

We are left with the forth case.
Solve the equation and obtain
\begin{equation}\label{eqn:Uy2-fourth-case}
  U_y^2=-\frac{1}{x}\,U_x\,(1+x\,U_x).
\end{equation}
Then substitute this into the $p_1^4$-component of \eqref{eqn:key-block} and obtain an expression for $U_{xx}$, and from the integrability condition we obtain an expression for $U_{xy}$.\footnote{One might want to check that both expressions are compatible, which is true.}
The other components of \eqref{eqn:key-block} are then satisfied trivially.
At this point, it is a good idea to go back to the expressions for $U_y^2$ and $U_{xy}$.
Combining both, we find the equation
\[
  \frac{\mathrm{d}}{\mathrm{d}x} U_y = -4x\,U_y^3
\]
which is an ODE for $U_y$ and can be solved in a straightforward way. 
The solution is
\[
  U_y = \frac{1}{\sqrt{ (4x^2-f_1(y)) }}
\]
Use this to replace $U_y^2$ in \eqref{eqn:Uy2-fourth-case}:
\[
  f_1(y) = -\frac{ x(1+4x\,U_x+4x^2\,U_x^2) }{ U_x\,(1+x\,U_x) }
\]
Solve this for $U_x$. There are two branches of possible solutions:
\[
  U_x= \frac{1}{2}\,\frac{-4x^2-f_1 \pm\sqrt{4x^2\,f_1+f_1^2}}{x(4x^2+\,f_1)}.
\]
We can use the integrability criterion to find an explicit form for $f_1$.
First, obtain two differential equations:
\[
  (f_1)_y \pm 4\sqrt{f_1} =0
\]
Up to the sign of the integration constant, both have the same solution
\[
   f_1=(2y+c)^2 =4y^2+4cy+c^2.
\]
Using again the equation for $U_y^2$ in terms of $U_x$ and integrating, one finds
\[
   U= \frac{1}{2}\,\ln(2y+c+\sqrt{4x^2+4y^2+4yc+c^2}) +f_2(x)
\]
with $f_2$ first being an unspecified integration `constant'.
Checking if this solution is compatible with the expression for $U_x$ found above, $(f_2)_x$ can take two possible values:
Either $(f_2)_x=-\frac{1}{x}$, or
\[
   \frac{df_2}{dx}=(f_2)_x(x)=\frac{-4x}{(2y+c+\sqrt{4x^2+(2y+c)^2})\,\sqrt{4x^2+(2y+c)^2}}.
\]
Now, consider the \name{Ernst} equation $U_x+x\,U_{yy}+x\,U_{xx}=0$. For the first solution for $(f_2)_x$, this implies $x=0$, so this is no valid solution.
Therefore, conclude
\[
  U= \frac{1}{2}\,\ln(2y+c+\sqrt{4x^2+4y^2+4yc+c^2})-\ln(x)+c_2
\]
with an additional integration constant $c_2\in\mathds{R}$.
By lemma~\ref{la:flat_metric_parametrisation}, the metric is flat and therefore all Killing tensor fields are reducible. This concludes step (vi) in the list of section~\ref{sec:method}.

To complete the proof, we still have to account for the case when \eqref{eqn:seitenzweig} is not satisfied. In this case, either $U_x=-\frac{1}{x}$ (this is covered by lemma \ref{la:Ux(x)-means-Uy=0}) or
\[
  U_x\,(1+x\,U_x)-3x\,U_y^2 =0
\]
We solve for $U_y^2$:
\[
  U_y^2 = \frac{U_x\,(1+x\,U_x)}{3x}
\]
From the $p_1^4$-component and the integrability criterion, we can also get another expression for $U_y^2$:
\[
  U_y^2 = \frac{3\,U_x\,(1+x\,U_x)}{x}
\]
The only way to allow both solutions to be true is if $U_x=0$ or $U_x=-\frac{1}{x}$. Both cases are covered by lemma~\ref{la:Ux(x)-means-Uy=0}.
\end{proof}

We have considered odd-parity third-degree integrals in a StAV space-time. Let us now summarize the results and merge them into one theorem:
\begin{proof}[Proof of theorem 1]
If $M$ is flat on a neighborhood, then it is totally reducible there \cite{thompson_killing_1986}. Thus assume that $M$ is non-flat.\\
First, consider only odd-parity integrals:\\
\textbf{Claim:}
Let $M$ be non-flat with $\omega=0$, $\nabla\,V^{\phi\phi}\not=0$, and $\nabla\,V^{tt}\not=0$.
Let $F$ be a third-degree integral of odd-parity in $M$.
Then $F$ is reducible by at least one degree.\\
\textit{Proof of the claim.}
First, let us consider the case when there is an additional Killing vector field. As we have seen in lemma \ref{la:rk1}, this means that the odd-parity third-degree integral is reducible by the (linear) integral $p_y$. So, the assertion is proven in this case.

Second, if there is no additional Killing vector field, proposition \ref{la:rk2} tells us (provided $\alpha\not=0$) that there is no odd-parity third-degree integral. In the case $\alpha=0$, we have $F=0$. Thus, the assertion is proven.\qed
\vskip 0.3cm
Now, consider the even-parity contributions.
The quadratic contributions $F^\circled{2}_{\phi}$ and $F^\circled{2}_{t}$ must obey the equation
\begin{align*}
\{T,F^\circled{2}_k\}	&= 0 \\
\shortintertext{as well as equations of the form}
\{T,F^\circled{0}_{abk}\}+\{V^{ab},F^\circled{2}_k\}	&= 0,
\end{align*}
where $a,b,k\in \{\phi,t\}$.

These, however, are precisely equations \eqref{eqns:quadratic-integrals} for quadratic integrals with leading term $F^\circled{2}_\phi$ or $F^\circled{2}_t$, respectively.
This means that
\[
  p_a\,(F^\circled{2}_a
  			+F^\circled{0}_{a\phi\phi}\,p_\phi^2
				+F^\circled{0}_{a t\phi}\,p_t\,p_\phi
				+F^\circled{0}_{a tt}\,p_t^2),
\]
$a=\phi,t$, are quadratic integrals and therefore, the even-parity contributions to the degree~3 integral $F$ are reducible by $p_\phi$ and $p_t$, respectively.
Hence, also the entire integral $F$ (consisting of the parts with degree from 3 down to 0) is reducible by one degree.
\end{proof}

\subsection{Zipoy-Voorhees}
Consider again the \name{Zipoy-Voorhees} class. We already considered third-degree odd-parity integrals in such space-times. Let us now consider even-parity components and assume w.l.o.g.\ $\delta\not=0$.
We use for the Hamiltonian~$H$ the representation
\begin{equation}
  H=\Omega_1\,p_x^2+\Omega_2\,p_y^2+V_\phi\,p_\phi^2+V_t\,p_t^2
\end{equation}
and denote the integral by
\[
  F=a_0\,p_x^2+a_1\,p_x\,p_y+a_2\,p_y^2 + b_0\,p_\phi^2+b_1\,p_\phi\,p_t+b_2\,p_t^2.
\]

From each polynomial of degree~1 after split w.r.t.\ $p_x,p_y$ (cf.\ remark~\ref{rmk:quadratic-integrals}), we obtain integrability conditions for $b_0$ and $b_2$. Automatically, $b_1=\const$ is no longer of interest.

Combining the \name{Bertrand-Darboux} relations and equations obtained from the degree~3 polynomial after splitting w.r.t.\ $(p_x,p_y)$, we can solve for derivatives of $a_0$, $a_1$ and $a_2$, and derive integrability conditions for them.
From the integrability conditions, we can deduce that $a_1=0$ and that (at least if $\delta\not=1$)
\[
  (y^2-1)a_2+(x^2-1)a_0=0.
\]
From the \name{Bertrand-Darboux} equations for $b_0$, $b_2$, we can now deduce $d(a_0)$ in terms of $a_0$ and solve the system of differential equations, obtaining
\[
  a_0=c_1\,(y^2-x^2)^{1-\delta^2}(x+1)^{\delta^2+\delta-1}(x-1)^{\delta^2-\delta-1}.
\]
Then, we can immediately compute $a_2$:
\[
  a_2=-c_1\,\left(\frac{x^2-1}{x^2-y^2}\right)^{\delta^2}\,\left(\frac{x+1}{x-1}\right)^\delta\,\frac{x^2-y^2}{y^2-1}.
\]
Finally, from the equations obtained from the degree~1 polynomial after split w.r.t.\ $p_x,p_y$, we obtain the derivatives of $b_0$, $b_2$, and by integration
\begin{align*}
  b_0		&= -c_1\,(y^2-1)\,(x^2-1)\,\left(\frac{x+1}{x-1}\right)^\delta+c_2 \\
  b_2		&= -c_1\,\left(\frac{x-1}{x+1}\right)^\delta+c_3
\end{align*}
Comparing this result to the Hamiltonian shows that 
\[
  F=c_1\,H+c_2\,p_\phi^2+c_3\,p_\phi\,p_t+c_4\,p_t^2.
\]
This means that every quadratic integral is reducible, provided $\delta\not=1$ (in case $\delta=1$, we obtain the Schwarzschild metric, which is integrable with the additional integral in involution being given by a quadratic integral. This quadratic integral, however, is reducible by linear integrals that are not in involution).
Together with theorem~\ref{thm:main-result-1}, this proves the assertion.
\qed

\section{Conclusion}
In this paper we gave a proof for the reducibility of valence~3 Killing tensor fields in static and axially symmetric vacuum space-times (Weyl's class).
We saw that using prolongation-projection is an efficient way to decide on the existence of integrals in SAV metrics even if the metric is not given specifically.
We plan to extend the result for degree~3 to the fully stationary SAV case. Though computationally more challenging, we are not aware of major conceptual problems with this.
As for generalizations beyond the SAV context, the line of reasoning made here is in principle not specific to the SAV class of space-times, and an analogous approach might work for other classes of space-times, too.

\section*{Acknowledgments}
The author is financially supported by Deutsche Forschungsgemeinschaft (research training group 1523/2 Quantum and Gravitational Fields).
He wishes to thank Vladimir Matveev for discussions and for remarks on the manuscript.

\printbibliography

\noindent The computer algebra computations for this paper have been performed using Maple 18. Maple is a trademark of Waterloo Maple Inc.

\end{document}